\documentclass[12pts]{amsart}
\usepackage[margin=.8in]{geometry}
\usepackage{tikz}
\usepackage{subcaption}
\usepackage{graphicx, color}
\usepackage{fancyhdr, multicol}
\usepackage{float}
\usepackage{hyperref}
\usepackage{amsmath, amsthm, amsfonts, amssymb}
\usepackage[parfill]{parskip}

\newcommand{\R}{\mathbb{R}}

\DeclareMathOperator{\Sym}{Sym}
\DeclareMathOperator{\tr}{tr}

\newcommand{\Pos}{\mathcal{P}}

\newcommand{\ideal}{\mathcal{I}}

\newcommand{\pard}[2]{\frac{\partial #1}{\partial #2}}
\newcommand{\change}[1]{{#1}}

\author{Grigoriy Blekherman}
\address{School of Mathematics, Georgia Institute of Technology,
686 Cherry Street
Atlanta, GA 30332}\email{greg@math.gatech.edu}

\author{Kevin Shu}
\address{School of Mathematics, Georgia Institute of Technology,
686 Cherry Street
Atlanta, GA 30332}\email{kshu8@gatech.edu}

\title{Sums of Squares and Sparse Semidefinite Programming}
\thanks{The authors were partially supported by NSF grant DMS-1901950.}

\newtheorem{theorem}{Theorem}
\newtheorem*{theorem*}{Theorem}
\newtheorem{corr}[theorem]{Corollary}
\newtheorem{remark}[theorem]{Remark}
\newtheorem{lemma}[theorem]{Lemma}
\newtheorem{example}[theorem]{Example}
\newtheoremstyle{named}{}{}{\itshape}{}{\bfseries}{.}{.5em}{Theorem (#3)}
\theoremstyle{named}

\begin{document}
\begin{abstract}
We consider two seemingly unrelated questions: the relationship between nonnegative polynomials and sums of squares on real varieties, and sparse semidefinite programming. This connection is natural when a real variety $X$ is defined by a quadratic square-free monomial ideal. In this case  nonnegative polynomials and sums of squares on $X$ are also natural objects in positive semidefinite matrix completion. We show quantitative results on approximation of nonnegative polynomials by sums of squares, which leads to applications in sparse semidefinite programming.
\end{abstract}
\maketitle
\section{Introduction}
The relationship between nonnegative polynomials and sums of squares is a cornerstone problem in real algebraic geometry. In a fundamental 1888 paper Hilbert showed that there exist globally nonnegative polynomials that are not sums of squares of polynomials \cite{MR1510517}. He then showed in 1893 that a bivariate globally nonnegative polynomial is a sum of squares of rational functions. Hilbert's 17th problem asked to show that this was true for any number of variables. It was solved by Artin in 1920's using the Artin-Schreier theory of real closed fields \cite{BCR}. In the last twenty years, questions about nonnegative polynomials and sums of squares (of polynomials) gained new prominence due to the development of sums of squares relaxations in optimization \cite{BPT}. Testing whether a nonnegative polynomial is a sum of squares can be done with \textit{semidefinite programming}, and sums of squares relaxations currently achieve best-known results for a variety of optimization problems \cite{barak2015dictionary, hopkins2016fast}. 

Hilbert's theorem on equality has been generalized to the setting of sums of squares on real projective varieties, resulting in a clearer classification of cases \cite{free_resolution,Blekherman2015SumsOS}. Another approach is to ask the more quantitative question of whether the sums of squares polynomials are a `good approximation' of nonnegative polynomials. It was shown in \cite{Blekherman2006} that the normalized volume of an appropriate slice of the cone of nonnegative polynomials is asymptotically larger than that of the sums of squares polynomials if the degree is fixed and the number of variables tends to infinity. In this paper we initiate a quantitative study of sums of squares approximation of nonnegative polynomials on real projective varieties. 

\subsection{Sums of Squares and Matrix Completion}
In our setting, the main algebraic objects are quadratic square-free monomial ideals, and the corresponding varieties are unions of coordinate subspaces. We now describe how these ideals and their associated varieties are related to the \textit{positive semidefinite matrix completion problem} (see \cite{free_resolution} for more on monomial ideals, and \cite{Laurent2009} for a survey on some results in matrix completion).

Given a graph $G$, we consider a square-free monomial ideal $\ideal_G$ whose monomials are non-edges of $G$: $\ideal_G = \langle x_i x_j : \{i,j\} \not \in E(G) \rangle$. The variety $X_G$ defined by $\ideal_G$ is the union of linear subspaces corresponding to maximal cliques of $G$. Under the standard basis, quadratic forms on $X_G$ correspond to the \textit{partially specified symmetric matrices} where the diagonal entries and the off-diagonal entries corresponding to the edges of $G$ are specified, and the rest of the entries are unspecified. 

\begin{example}
Consider the 5-vertex cycle graph $C_5$. We think of the edges as indexing entries in a $5 \times 5$ matrix (see Figure \ref{fig:partial}). We associate to $C_5$ a class of `partially specified matrices', where we think of entries indexed by edges as being given, and the remainder as being unknown. Our goal will typically be to complete these matrices by setting the unknown entries so that the resulting matrix is positive semidefinite (PSD).

\begin{figure}
\centering
\begin{subfigure}{0.5\textwidth}
  \centering
	\begin{tikzpicture}
        \node[circle,fill=black,minimum size=0.25mm](v1) at (18:1) {};
        \node[circle,fill=black,minimum size=0.25mm](v2) at (90:1) {};
        \node[circle,fill=black,minimum size=0.25mm](v3) at (162:1){};
        \node[circle,fill=black,minimum size=0.25mm](v4) at (234:1){};
        \node[circle,fill=black,minimum size=0.25mm](v5) at (306:1){};
        \node[circle,fill=black,minimum size=0.25mm](v6) at (378:1){};
        \draw (v1) -- (v2) -- (v3) -- (v4) -- (v5) -- (v6) -- cycle;
	\end{tikzpicture}
  \caption{The 5-cycle graph $C_5$.}
  \label{fig:sub1}
\end{subfigure}%
\begin{subfigure}{0.5\textwidth}
  \centering
\[
\begin{pmatrix}
1 & 1 & ? & ? & 1\\
1 & 1 & 1 & ? & ?\\
? & 1 & 1 & 1 & ?\\
? & ? & 1 & 1 & 1\\
1 & ? & ? & 1 & 1\\
\end{pmatrix}
\]
  \caption{A partially specified matrix corresponding to $C_5$. $?$ corresponds to an unknown entry.}
  \label{fig:sub2}
\end{subfigure}
  \caption{The partially specified matrix here represents the quadratic form $\sum_i x_i^2 + 2\sum_{\{i,j\}\in E(C_n)}x_ix_j$ on the variety $X_{C_n}$. There are many quadratic forms on $\R^n$ which restrict to the same form on $X_{C_n}$, for instance $\left(\sum_i x_i\right)^2$, which is a sum-of-squares completion of this form.}
\label{fig:partial}
\end{figure}

\end{example}
Nonnegative quadratics on $X_G$ correspond to partially specified matrices where all fully specified submatrices are positive semidefinite. Sum of squares quadratics on $X_G$ correspond to matrices that can be completed to a full positive semidefinite matrix. By a result of Grone et al. in \cite{GRONE1984109}, all nonnegative quadratics on $X_G$ are sum of squares if and only if $G$ is a chordal graph.

We find a class of graphs $\mathcal{G}$ for which we can quantify the extent to which sums of squares provide a good approximation of nonnegative quadratics. The class $\mathcal{G}$ contains chordal and series-parallel graphs and is closed under the operations of taking clique sums and joining with complete graphs (see Section \ref{sec:resultdetail} for a precise description). The quality of approximation of nonnegative polynomials by sums of squares for $G\in \mathcal{G}$ depends only on the size of the smallest induced cycle of size at least 4 in $G$ and the number of vertices of $G$. For $n$-cycles $C_n$ we show that, surprisingly, the quality of approximation improves sharply as $n$ increases, verifying an experimental observation of Drton and Yu \cite{drtonyu}.

A similar class of graphs was studied in \cite{MR1342017} in relation with the positive semidefinite matrix completion problem. A graph is called \emph{cycle completable} if positive semidefiniteness of fully specified submatrices, and completability of each induced cycle to a positive semidefinite submatrix is sufficient for the existence of a full positive semidefinite completion. It was shown in \cite{MR1342017} that cycle completable graphs are precisely the graphs which can be formed by taking clique sums of complete graphs and series parallel graphs. This class of graphs can also be characterized by forbidding the wheel graph and its splittings as subgraphs. Since the wheel graph is a join of a cycle graph and
\change{the}
one vertex graph, the graph class $\mathcal{G}$ is strictly larger than cycle completable graphs. Some computational and graph-theoretic properties of the class $\mathcal{G}$ are discussed in Section \ref{sec:G}.

\subsection{Semidefinite Programming}
Our results have direct application to semidefinite programming. Consider a general semidefinite program of the form
\begin{equation*}
\begin{aligned}
    \text{minimize} &&\langle B^0, X\rangle\\
    \text{such that } &&\tr(X) = 1,\\
                      &&\langle B^1, X\rangle = b_1,\\
                      &&\langle B^2, X\rangle = b_2,\\
                      &&\dots\\
                      &&\langle B^k, X\rangle = b_k,\\
                  &&  X \succeq 0\change{.}
\end{aligned}
\end{equation*}
In some circumstances, not all of the entries of $X$ are needed to define the constraints or the objective, which corresponds to cases when some of the entries of the $B^{\ell}$ are zero. Explicitly, let $G$ be the graph where $\{i,j\} \in E(G)$ if some $B^{\ell}_{i,j} \neq 0$. If $\{i,j\} \not \in E(G)$, then the linear constraints do not depend on the value of $X_{i,j}$, and this sparsity should be exploitable.

It is known that if the graph $G$ is chordal, then it is possible to decompose the semidefinite program into smaller semidefinite programs which can be solved faster, an observation which seems to originate from \cite{FUKUDA1970}.
This idea is extended greatly in \cite{chordal}. Reducing a problem to one with a chordal sparsity pattern may be hard, as the problem of finding the sparsest possible chordal graph containing a given one is NP-hard to approximate \cite{CAO2020104514}.
We extend this idea to the case of possibly \textit{nonchordal} graphs in our class $\mathcal{G}$, by showing that it is possible to decompose a semidefinite program with a sparsity pattern corresponding a graph $G \in \mathcal{G}$, with an error depending only on the smallest induced cycle of $G$ and the number of vertices in $G$.

In quantitative terms, for any graph $G\in \mathcal{G}$, \textit{we can bound the multiplicative error for this decomposition of the semidefinite program} by
{\color{black}$1+O\left(\frac{n}{g^3}\right)$ where $g$}
is the size of the smallest induced cycle in the graph of length at least $4$, assuming some technical conditions.
In particular, these results apply to the MAX-CUT semidefinite program of Goemans and Williamson for graphs in $\mathcal{G}$. \cite{goemans1995improved}

The value of this decomposition is that the smaller semidefinite programs produced are frequently easier to solve, and require fewer variables in total than the original semidefinite program.
The reduction in the number of variables can mean the difference between the program being solvable and being unsolvable in practice, so our results may make a larger class of semidefinite programs practically tractable.

\subsection{Results in Detail}\label{sec:resultdetail}
\subsubsection{Definitions}
Let $G$ be a graph. We can associate to $G$ a quadratic square-free monomial ideal $\ideal_G$ where the monomials in $\ideal_G$ correspond to the non-edges of $G$. The ideal $\ideal_G$ is a radical ideal and the variety $X_G$ defined by $\ideal_G$ is the union of coordinate subspaces which correspond to maximal cliques of $G$. More explicitly, if we let $C_G$ be the set of all maximal cliques of $G$, then
\[
    X_G = \bigcup_{K \in C_G} \text{span}(e_i : i \in K),
\]
where the $e_i$ are coordinate vectors, and $\text{span}$ denotes the linear span of a collection of vectors. We let $R_G=\mathbb{R}[x_1,\dots,x_n]/\ideal_G$ be the graded coordinate ring of $X_G$. 

We let $\mathcal{P}(G)$ be the cone of quadratic forms in $R_2$, (the degree $2$ part of $R_G$), which are nonnegative on $X_G$, and we let $\Sigma(G)$ be the cone of quadratic forms in $R_2$ that are sums of squares of linear forms. Both $\Sigma(G)$ and $\mathcal{P}(G)$ are full-dimensional, convex, pointed cones in $R_2$ \cite[Chapter 4]{BPT}.

For a form $q \in \mathcal{P}(G)$ we define $\varepsilon_G(q)\in \mathbb{R}$ to be the smallest real number such that 
\[
   q(x) + \varepsilon_G(q)  \tr (q) (x_1^2+\dots+x_n^2) \in \Sigma(G) \label{approx}, 
\]
where $\tr(q) = \sum_{i=1}^n q(e_i)$.
If $\varepsilon_G(q)$ is small, then there is a sum of squares polynomial which closely approximates $q(x)$.
We let $\varepsilon(G)$ be the largest value of $\varepsilon_G(q)$ over all $q\in \mathcal{P}(G)$. Observe that $\varepsilon(G)\geq 0$ for any graph $G$ and $\varepsilon(G)=0$ if and only if $G$ is a chordal graph. 
We call $\varepsilon(G)$ the \emph{(trace-normalized) conical distance}  between $\mathcal{P}(G)$ and $\Sigma(G)$.
The quantity $\varepsilon(G)$ appears naturally in several contexts. See Section \ref{sec:condis} for a more detailed discussion and motivation for this definition.

If the conical distance between $\mathcal{P}(G)$ and $\Sigma(G)$ is small, then the above discussion shows that sums of squares quadratics on $X_G$ provide a good approximation of nonnegative quadratics on $X_G$. However, using conical distance as a measure of approximation quality has its limitations. Lemma \ref{lemma:subgraph} shows that if $G$ contains an induced four-cycle $C_4$, then the conical distance between $\mathcal{P}(G)$ and $\Sigma(G)$ will be large (i.e. bounded from below by a constant). 
When the conical distance is large, the results in this paper will not provide strong approximation results for semidefinite programming, though other techniques may still result in useful bounds.

To state our results, we will also require some graph theoretic terminology. A minor of a graph $G$ is any graph which can be obtained from $G$ by iteratively deleting edges, or combining the endpoints of an edge to form a single vertex. A series parallel graph is any graph without the 4-vertex complete graph $K_4$ as a minor. We say that a graph $S$ is the \textbf{clique sum} of graphs $G$ and $H$ if $S$ is obtained from $G$ and $H$ by finding isomorphic cliques in $G$ and $H$ and combining $G$ and $H$ by identifying corresponding vertices contained in the clique. 
We denote the clique sum of $G$ and $H$ by $G \oplus H$.
See Section 3.1 for a precise definition.  As a note, in much of the literature, when taking a clique sum of two graphs, it may be allowed to delete some of the edges contained in the clique. In this paper, we follow the convention in \cite{JOHNSON1996155}, which uses the term \textit{clique identification} where we use clique sum.

The \textbf{cone} over a graph $G$ is obtained by introducing a single new vertex to $G$, and then adding edges from that new vertex to every vertex in $G$. We denote the cone over $G$ by $\hat{G}$. See section 3.2 for a more detailed definition.
 
\subsubsection{Main Results}
Our main result is a class of graphs for which we can compute the conical distance exactly.
Let $\mathcal{G}$ be the smallest collection of graphs satisfying the following three conditions: (1) Series-parallel and chordal graphs are in $\mathcal{G}$ (2) Any clique sum of two graphs in $\mathcal{G}$ is in $\mathcal{G}$ and (3) The cone over any graph in $\mathcal{G}$ is in $\mathcal{G}$. 
{\color{black}
\begin{theorem}[Theorem \ref{thm:main}]
    Let $G \in \mathcal{G}$. Then $$\varepsilon(G) = \frac{1}{g}\left(\frac{1}{\cos(\frac{\pi}{g})} - 1  \right)\change{,}$$ where $g$ is the length of the smallest induced cycle of $G$ of length at least $4$.
\end{theorem}
}
In particular, this gives an exact computation for the conical distance of the wheel graph $W_n=\hat{C}_n$, so the set $\mathcal{G}$ is strictly larger than the cycle-completable graphs.

We prove this result by computing the conical distance in three key cases.

We first show that $\varepsilon(G)$ behaves well under the clique sum and cone operations. These are natural operations to consider because they preserve the class of chordal graphs, which are the only cases where $\varepsilon(G) = 0$. This results in the two theorems:

\begin{theorem*}[Theorem \ref{thm:csum}]
    $\varepsilon(G \oplus H) = \max\{\varepsilon(G), \varepsilon(H)\}$.
\end{theorem*}

\begin{theorem*}[Theorem \ref{thm:join}]
    $\varepsilon(\hat{G}) = \varepsilon(G)$.
\end{theorem*}

Finally, we compute $\varepsilon(G)$ when $G$ is an $n$-cycle $C_n$.

\begin{theorem*}[Theorem \ref{thm:cycle}]
    For $n > 3$,
    $$\varepsilon(C_n)=\frac{1}{n}\left(\frac{1}{\cos(\frac{\pi}{n})} - 1  \right).$$
\end{theorem*}
Theorem \ref{thm:cycle} implies that the ratio of the volumes between appropriately normalized slices of $\Sigma(C_n)$ and $\mathcal{P}(C_n)$ approaches $1$ as $n$ goes to infinity\change{,} verifying an experimental observation of Drton and Yu \cite{drtonyu}.

We now discuss applications of our results to semidefinite programming. Unlike the usual intuition for nonnegative polynomials and sums of squares, the cone $\mathcal{P}(G)$ is simpler algorithmically in some cases, as it corresponds to partially specified matrices all of whose fully specified principal submatrices are positive semidefinite. In semidefinite programming problems we often encounter the cone $\Sigma(G)$, and we would like to replace $\Sigma(G)$ with the simpler cone $\mathcal{P}(G)$. This leads to replacement of the global positive semidefiniteness constraint with a series of smaller constraints, which is very advantageous in practice. Our bounds on the quality of approximation of $\mathcal{P}(G)$ by $\Sigma(G)$ allows us to show that for our class $\mathcal{G}$ this substitution has a controlled cost.

More specifically, consider the conical program
\begin{equation*}
\begin{aligned}
    \text{minimize} &&\langle B^0,X\rangle\\
    \text{such that } &&\tr(X) = 1,\\
                      &&\langle B^1,X\rangle = b_1,\\
                      &&\dots\\
                      &&\langle B^k,X\rangle = b_k,\\
                      &&  X \in K\change{.}
\end{aligned}
\end{equation*}
\change{W}hen $K = \Sigma(G)$, we denote the value of this program as $\alpha$\change{,} and when $K = \mathcal{P}(G)$, we denote the value by $\alpha'$.

Under some mild technical conditions, we obtain the following bound:
{\color{black}
\begin{theorem*}[Theorem \ref{thm:SDP}]
    \[
        \alpha' \le \alpha \le \frac{1}{1+n \varepsilon(G)}\alpha' + \frac{\varepsilon(G)}{1+n\varepsilon(G)}tr(B^0)\change{.}
    \]
\end{theorem*}
}
This is especially effective for a rescaled version of the Goemans-Williamson MAX-CUT semidefinite program, as several of these technical assumptions are satisfied automatically.
Thus we can extend the ideas in \cite{chordal} to nonchordal graphs.

Finally, we give some graph theoretic results about our graph class $\mathcal{G}$, in particular showing that it is recognizable in polynomial time, and that the number of maximal cliques in a graph $G \in \mathcal{G}$ is polynomial in the size of $G$ in Section \ref{sec:graph_properties}.

\section{Definitions and Problem Setup}
\subsection{Graphs and Partially Specified Matrices}
Let $G = (V, E)$ be an undirected graph on $n$ vertices. 
We use $V(G)$ and $E(G)$ to denote the set of vertices and edges of $G$ respectively.

We associate to $G$ a vector space, $\Sym(G)$, which we think of as assigning a real number to each edge and each vertex of $G$
\[
    \Sym(G) = \R^{E(G) + V(G)}.
\]

Choosing an ordering of the vertices of $G$ identifies $V(G)$ with $[n]$, so that we can regard each edge in $G$ as an unordered pair of integers. Under this identification, the edges of $G$ index off-diagonal entries in a symmetric $n\times n$ matrix, and the vertices index the diagonal entries.

\begin{example}
Let $K_n$ be the $n$-vertex complete graph on the vertex set $[n]$. $\Sym(K_n)$ is isomorphic to the space of symmetric $n\times n$ matrices over $\R$. We will use $\Sym(K_n)$ to denote the space of symmetric matrices throughout the paper.
\end{example}

For any $G$ with $n$ vertices, there is a linear projection
\[
    \pi_G : \Sym(K_n) \rightarrow \Sym(G),
\]
given by sending a matrix $M$ to the vector of entries corresponding to the edges of $G$:
\[
    \pi_G(M) = (M_{ij})_{\{i,j\} \in E(G)+V(G)}.
\]

We can therefore think of $\Sym(G)$ as a collection of \textbf{partially specified matrices}, where we start with a matrix $M$, and then `forget' the entries of $M$ that do not correspond to edges of $G$. Note that symmetric matrices in $\Sym(K_n)$ correspond naturally to quadratic forms over the variety $\R^n = X_{K_n}$, and similarly $\Sym(G)$ corresponds naturally to the space of quadratic forms over the variety $X_G$. The projection $\pi_G$ on the level of quadratic forms is induced by the embedding of $X_G$ inside $X_{K_n}$.

If $A$ is a partially specified matrix in $\Sym(G)$, then we say that a matrix $M \in \pi_G^{-1}(A)$ is a \textbf{completion} of $A$, since it is a fully specified matrix which agrees with $A$ on all entries in $A$ which are specified.

More generally, if $H$ is a subgraph of $G$, then we also have a projection 
\[
    \pi_H : \Sym(G) \rightarrow \Sym(H),
\]
which additionally forgets the entries of a partially specified matrix which correspond to edges not in $H$. Although technically, this notation does not distinguish between the projection from $\Sym(K_n)$ to $\Sym(H)$ and the projection from $\Sym(G)$ to $\Sym(H)$, this distinction will not be relevant in this paper, so we will abuse notation and write $\pi_H$ without specifying the domain.

Suppose that $K \subseteq G$ is a clique contained in $G$, then in this case, if $A$ is a partially specified matrix in $\Sym(G)$, then $\pi_K(A)$ will be a fully specified principle submatrix of $A$.

\subsection{Partially Positive and Partially Sums of Squares Matrices}
Inside the space of symmetric matrices, $\Sym(K_n)$, we have the subset of positive semidefinite matrices. We will denote these by $\Sigma(K_n) = \mathcal{P}(K_n)$, as they correspond naturally to sum of squares (equivalently nonnegative) quadratic forms over $\R^n$. 
Now, we can define the subset of $\Sym(G)$ which are the projections of positive semidefinite (PSD) matrices in $\Sym(K_n)$. Precisely, 
\[
    \Sigma(G) = \pi_G(\Sigma(K_n)).
\]
Equivalently, we can think of $\Sigma(G)$ as being those partially specified matrices that can be completed to a PSD matrix. We call these \textbf{ partially sums of squares matrices}. These also correspond to quadratic forms on $X_G$ which are sums of squares.

Now, we note that if $M$ is a PSD matrix which completes $A$, then any principal submatrix of $M$ is also PSD, so for any subgraph $H$ of $G$, $\pi_H(A)$ is PSD completable. In particular, if $K$ is a clique contained in $G$, the fully specified submatrix $\pi_K(M) = \pi_K(A)$ is PSD. Therefore we see that if $A \in \Sigma(G)$ is PSD completable, and $K$ is a clique contained in $G$, then $\pi_K(A) \in \Sigma(K)$. In particular $\pi_K(A)$ is a fully specified PSD matrix.

Consider the convex cone 
\[
    \mathcal{P}(G) =  \{A \in \Sym(G)\,\,  : \,\,  \pi_C(A) \in \mathcal{P}(C) \,\,\, \text{for all}\,\,\, C \in C_G \}.
\]
These are the matrices in $\Sym(G)$ such that all of their fully specified principal submatrices are PSD, and these correspond to nonnegative quadratic forms on $X_G$. We will refer to a partially specified matrix in $\mathcal{P}(G)$ as a \textbf{partially positive matrix}. It is clear that $\Sigma(G) \subseteq \mathcal{P}(G)$.

\subsection{Conical Distance}\label{sec:condis}
Finding a way to quantify the difference between $\mathcal{P}(G)$ and $\Sigma(G)$ requires careful consideration. Part of the difficulty lies in the fact that these two convex sets are not compact, which is a requirement of some typical notions of distance in convex geometry. In order to make these cones compact, we choose a hyperplane so that the intersection of these cones with that hyperplane is compact, and then use one of these more common notions of convex analysis. Fortunately, in this case, it is natural to consider the hyperplane of trace-1 matrices, and identify a multiple of the identity matrix with the origin.

Let $H$ be the hyperplane in $\Sym(G)$ given by $H = \{A \in \Sym(G) : \tr(A) = 1\}$ (this is well defined, since the ideal $\ideal_G$ is square-free).
Let $\bar{\Sigma}(G)$ and $\bar{\mathcal{P}}(G)$ be the intersections of $\Sigma(G)$ and $\mathcal{P}(G)$ with $H$ respectively. 
We define the conical distance between $\mathcal{P}(G)$ and $\Sigma(G)$ to be
{\color{black}
\[
    \varepsilon(G) = \min_{\varepsilon} \{\varepsilon : \text{for all} \,\, M \in 
    \bar{\mathcal{P}}(G) \,\,\,  M +  \varepsilon I_n \in \Sigma(G) \}.
\]
It is clear that if $\varepsilon \ge \varepsilon(G)$, then for any $A \in \bar{\mathcal{P}}(G)$, $A + \varepsilon I_n \in \Sigma(G)$.
}

This definition of conical distance has connections to several familiar notions. We will make particular note of the connection between conical distance and the notion of expansion ratio, and the connection between conical distance and the eigenvalues of completions of partial positive matrices.

\subsubsection{Expansion Ratio}
We can consider the \emph{expansion ratio} $\alpha(G)$ of $\bar{\Sigma}(G)$ and $\bar{\mathcal{P}}(G)$ with respect to the scaled identity matrix $\frac{1}{n}I_n$, which is a standard way of measuring distance between compact convex sets \cite{ball_convex}. The expansion ratio is defined to be the smallest $k > 0$ such that $k(\bar{\Sigma}(G)-\frac{1}{n}I_n)$ contains $\bar{\mathcal{P}}(G)-\frac{1}{n}I_n$.

{\color{black}
\begin{lemma}
    $\alpha(G) = 1+n\varepsilon(G)$.
\end{lemma}
\begin{proof}
   Let $k = 1+n\varepsilon(G)$ and consider any $M \in \bar{\mathcal{P}}(G)$. We see that
    \begin{align*}
        \frac{1}{k}(M - \frac{1}{n}I_n) + \frac{1}{n}I_n &= \frac{1}{k}(M + (k - 1)\frac{1}{n}I_n)\\
                                                       &= \frac{1}{k}(M + \varepsilon(G)I_n),
    \end{align*}
    where $M + \varepsilon(G) I_n \in \Sigma(G)$ by definition of $\varepsilon(G)$. 
    We also see that $\tr(\frac{1}{k}(M + \varepsilon(G)I_n)) = 1$.

    Therefore if $M \in \bar{\mathcal{P}}(G)$, there is an element of $\bar{\Sigma}(G)$, namely $X=\frac{1}{k}(M+\varepsilon(G)I_n)$ so that $k(X-\frac{1}{n}I_n) = M - \frac{1}{n}I_n$, i.e. $k(\bar{\Sigma}(G) - \frac{1}{n}I_n) \supseteq \bar{\mathcal{P}}(G)$. Thus, $\alpha(G) \le 1+n\varepsilon(G)$.

   The other bound $\alpha(G) \ge 1+n\varepsilon(G)$ is obtained in a similar way.
\end{proof}
}

\subsubsection{Eigenvalues}
We can also interpret $\varepsilon(G)$ in terms of the eigenvalues of the partially positive matrices:
\begin{lemma}
    {\color{black}If $A \in \bar{\mathcal{P}}(G)$, then $A$ can be completed to a matrix $M$ whose minimal eigenvalue is at least $-\varepsilon(G)$.}
\end{lemma}
\begin{proof}
    By definition, $\varepsilon(G)$ is the smallest $\varepsilon$ so that for every $A \in \bar{\mathcal{P}}(G)$, $A + {\color{black}\varepsilon} I_n$ is PSD completable. If $M$ is a PSD completion of {\color{black}$A + \varepsilon I_n$}, then {\color{black} $M - \varepsilon I_n$} is a completion of $A$ with minimal eigenvalue at least $-\varepsilon$. Similarly, if $M$ is a completion of $A$ with a minimum eigenvalue $\lambda_{\min}$, then $M - \lambda_{\min} I_n$ would be a PSD completion for $A - \lambda_{\min} I_n$. Thus,
\[
\varepsilon(G) = \min_{A \in \bar{\mathcal{P}}(G)} \max_{M \in \pi_G^{-1}(A)} -\lambda_{\min}(M).
\]
\end{proof}

In other words, computing the conical distance between this pair of cones is equivalent to finding completions of partially positive matrices with minimum eigenvalues which are as large as possible.

\subsection{Monotonicity}
One simple observation to make about conical distance is that it is monotonic increasing in the induced subgraph ordering, as indicated by the following lemma:
\begin{lemma}\label{lemma:subgraph}
If $H$ is an induced subgraph of $G$, then $\varepsilon(H) \le \varepsilon(G)$.
\end{lemma}

The remainder of this paper will be devoted to computing this distance for some classes of graphs.

\section{Constructions on Graphs}
We will show that the two operations of clique sums and joins of a graph with cliques act naturally on the conical distance of a graph.

\subsection{Clique Sums}
Suppose that $G$ and $H$ are graphs, and that for some $n$, there is a complete graph $K_n$ and injective homomorphisms $\phi : K_n \rightarrow G$ and $\psi:K_n \rightarrow H$.

We define the \textbf{clique sum of $G$ and $H$} (with respect to $\phi$ and $\psi$) to be the graph $S$ obtained by taking the disjoint union of the graphs $G$ and $H$ and then identifying $\phi(x)$ with $\psi(x)$ for each $x \in K_n$. This is the categorical pushforward of the two morphisms $\phi$ and $\psi$. For notational convenience, we will suppress the dependence of the clique sum on $\phi$ and $\psi$, and denote the clique sum of two graphs as $S = G \oplus H$.

We can consider $G$ and $H$ to be induced subgraphs of $G \oplus H$, so that the subgraph induced by $V(G) \cap V(H)$ is a clique. From Lemma \ref{lemma:subgraph}, it is clear that $\varepsilon(G \oplus H) \ge \max\{\varepsilon(G), \varepsilon(H)\}$. In fact, this inequality is an equality.
\begin{theorem}\label{thm:csum}
    $\varepsilon(G \oplus H) = \max\{\varepsilon(G), \varepsilon(H)\}$.
\end{theorem}
    For the proof, first recall that an $n\times n$ matrix $M$ is positive semidefinite if and only if there exist vectors $v_1, v_2, \dots, v_n \in \R^n$ such that $M$ is the \emph{Gram matrix} of $v_1, v_2, \dots, v_n$, i.e.  $M_{ij} = \langle v_i, v_j\rangle$ for all $i, j \le n$. From this fact, it is clear that the following lemma holds:
    \begin{lemma}
        If $A \in \Sym(G)$, then $A \in \Sigma(G)$ if and only if there are vectors $v_1, \dots, v_n \in \R^n$ so that for any $\{i,j\}\in V(G)+E(G)$, $\langle v_i, v_j \rangle = A_{ij}$.
    \end{lemma}
    We call such a collection of vectors a \emph{vector arrangement} inducing $A$.

    \begin{lemma}
        $A\in \Sigma(G \oplus H)$ if and only if $\pi_G(A) \in \Sigma(G)$ and $\pi_H(A) \in \Sigma(H)$.
    \end{lemma}
    \begin{proof}
    
    If $A \in \Sigma(G\oplus H)$, then it follows immediately that $\pi_G(A) \in \Sigma(G)$ and $\pi_H(A) \in \Sigma(H)$.
    
    Now, suppose that $\pi_G(A) \in \Sigma(G)$ and $\pi_H(A) \in \Sigma(H)$. We wish to show that $A \in \Sigma(G \oplus H)$. The idea will be to `glue' together a vector arrangement inducing $\pi_G(A)$ with one inducing $\pi_H(A)$ using an appropriate orthogonal transformation.
    
    For the sake of concreteness, assume that the vertices of $G \oplus H$ are $V(G\oplus H) = [n]$, and that the vertices of $G$ are $V(G) = [k]$, and finally that $V(H) = \{\ell,\ell+1,\dots,n\}$, where $\ell \le k$. This implies that the vertices in $V(G) \cap V(H)$ are $\{\ell, \dots, k\}$.

        Let $v_1, \dots, v_k \in \R^k$ be a vector arrangement inducing $\pi_G(A)$ and $w_{\ell}, \dots, w_n \in \R^{n-\ell+1}$ be a vector arrangement inducing $\pi_H(A)$. By considering $\R^k$ and $\R^{n-\ell+1}$ subspaces of $\R^n$, we can also take $v_1,\dots, v_k$ and $w_{\ell}, \dots, w_n$ to all lie in $\R^n$. We wish to produce a vector arrangement inducing $A$. 
        If $i,j \in V(G) \cap V(H)$, then $\{i,j\} \in E(G \oplus H)$, because we assumed that $V(G) \cap V(H)$ induced a clique in our definition of the clique sum. Therefore, from the definition of these vector arrangements, for each $i, j \in V(G) \cap V(H)$,
    \[
        \langle v_i, v_j \rangle = A_{i,j} = \langle w_i, w_j \rangle.
    \]

    Consider the subsets of the vectors, $S_1 = \{v_i : i \in V(G) \cap V(H)\}$ and $S_2 = \{w_i : V(G) \cap V(H)\}$. Our previous observation is equivalent to the fact that $\pi_{G \cap H}(A)$ is the Gram matrix of both $S_1$ and $S_2$. In particular, the Gram matrices of $S_1$ and $S_2$ are equal.

    This implies that there exists an orthogonal transformation $T \in O(n)$ so that $Tw_i = v_i$ for each $i \in V(G)\cap V(H)$. This is shown, for example in \cite[Chapter 7, Theorem 7.3.11]{horn2013matrix}.
    
    Consider the collection of vectors $u_i$ where $u_i = v_i$ if $i \le k$ and $u_i = Tw_i$ if $i \ge \ell$ (which is well defined, as $Tw_i = v_i$ when $\ell \le i \le k$). We claim that the $u_i$ induce $A$.  This follows from the fact that each edge $e \in E(G\oplus H)$ is either an edge of $G$ or an edge of $H$. If $e = \{i,j\}$ is an edge of $G$, then 
    \[
        \langle u_i, u_j \rangle = \langle v_i, v_j\rangle = A_{i,j}\change{.}
    \] 
    If $e$ is an edge of $H$, then by orthogonality of $T$,
    \[
    \langle u_i, u_j \rangle = \langle Tw_i, Tw_j\rangle = \langle w_i, w_j\rangle = A_{i,j}.
    \]
    
    In either case, we see that $\langle u_i, u_j \rangle = A_{i,j}$, as desired.
\end{proof}

\begin{proof}[Proof of Theorem \ref{thm:csum}]
Let $A$ be in $\bar{\mathcal{P}}(G \oplus H)$ and consider the matrix $A + \max\{\varepsilon(G), \varepsilon(H)\}I_n$.
        Consider the projection
    \[
        \pi_G(A + \max\{\varepsilon(G), \varepsilon(H)\}I_n) = \pi_G(A) + \max\{\varepsilon(G), \varepsilon(H)\}I_{|V(G)|}\change{.}
    \]
    Note that $\pi_G(A)$ is in $\mathcal{P}(G)$, so from the definition of  $\varepsilon(G)$, 
    \[
        \pi_G(A) + \varepsilon(G)\tr(\pi_G(A)) I_{|V(G)|} \in \Sigma(G).
    \]
   Further note that $\tr(\pi_G(A)) \le 1$, since $\tr(A) = 1$, and the diagonal entries of $A$ are nonnegative.
    Combining these two facts, we can see that 
    \[
        \pi_G(A) + \varepsilon(G) I_{|V(G)|} \in \Sigma(G),
    \]
    and therefore that $\pi_G(A) + \max\{\varepsilon(G), \varepsilon(H)\}I_{|V(G)|} \in \Sigma(G)$.

    A similar argument shows that $\pi_H(A + \max\{\varepsilon(G), \varepsilon(H)\}I_n)$ lies in $\Sigma(H)$. Thus, by our previous lemma, $A  + \max\{\varepsilon(G), \varepsilon(H)\}I_n$ is in $\Sigma(G\oplus H)$.
\end{proof}

\begin{remark}
A graph $G$ is chordal if and only if $G$ is the result of taking repeated clique sums of complete graphs (see, say, \cite{MR1171260}), and for a complete graph $K$, it is clear that the conical distance, $\varepsilon(K)$, is $0$, since $\Sigma(K_n) = \mathcal{P}(K_n)$.
\end{remark}

Thus, this immediately implies the known result that for any chordal graph $G$, $\varepsilon(G) = 0$. This, combined with the observation that any cycle $C_n$ of size at least 4 has $\varepsilon(C_n) > 0$, implies that chordal graphs are the only graphs for which $\mathcal{P}(G) = \Sigma(G)$. Thus, we obtain another proof of the fact that
\begin{corr}
    $\varepsilon(G) = 0$ if and only if $G$ is chordal.
\end{corr}

\subsection{Joins of Graphs with Cliques}
Another common construction in graph theory is the \textbf{join} of two graphs, $G$ and $H$. We define $G \land H$ by taking its vertex and edge sets to be
\[
    V(G \land H) = V(G) \sqcup V(H),
\]
\[
    E(G \land H) = E(G) \cup E(H) \cup \{\{i,j\} : i \in V(G), j \in V(H)\}.
\]

If $H$ is a one vertex graph, and $G$ is arbitrary, then define the \textbf{cone} over $G$ to be $\hat{G} = G \land H$. That is, $\hat{G}$ is the graph consisting of $G$, together with a new vertex, which we will denote by $v^*$, which is connected to all of the vertices in $G$. If $G$ is chordal, then $\hat{G}$ will also be chordal. 

For an arbitrary graph $G$, we see that $\hat{G}$ will have the same conical distance as $G$.
\begin{theorem}\label{thm:join}
    $\varepsilon(\hat{G}) = \varepsilon(G)$.
\end{theorem}
\begin{proof}

    The inequality that $\varepsilon(\hat{G}) \ge \varepsilon(G)$ follows from Lemma \ref{lemma:subgraph}, since $G$ is an induced subgraph of $\hat{G}$. 
    
    A partial matrix in $\Sym(\hat{G})$ can be regarded as a matrix in $\Sym(G)$, together with an appended row and column corresponding to the new vertex that is added. That is, a partial matrix in $\Sym(\hat{G})$ can be written as
    \[
        A = 
        \begin{pmatrix}
            Q & c \\
            c^{\intercal} & b
        \end{pmatrix}\change{,}
    \]
    where $Q \in \Sym(G)$, $c \in \R^{n}$ and $b \in \R$.

    Now, suppose that we have $A$ which is in $\mathcal{P}(\hat{G})$. We consider two cases: either $b = 0$ or $b > 0$.
    
    If $b = 0$, then consider a $2\times 2$ minor corresponding to an edge of the form $\{v^*, v\}$, for some $v \in V(G)$. We see that we will obtain a submatrix of the form
    \[
    \begin{pmatrix}
        Q_{vv} & c_{v} \\
        c_v & 0
    \end{pmatrix}.
    \]
    Since $A \in \mathcal{P}(G)$, all of the minors corresponding to edges must be positive semidefinite, and so we see that $c_v = 0$ for all $v \in V(G)$. This implies that $c = 0$, and in fact,
    \[
        A = 
        \begin{pmatrix}
            Q & 0 \\
            0 & 0
        \end{pmatrix}.
    \]
    Thus, we see that
    \[
        A +{\color{black}\varepsilon(G)I_n} = 
        \begin{pmatrix}
            Q +{\color{black}\varepsilon(G)I_{n-1}}& 0 \\
            0 & {\color{black}\varepsilon(G)}
        \end{pmatrix}.
    \]
    Clearly, because there is a PSD completion of $Q + {\color{black}\varepsilon(G)}I_n,$ we can obtain a PSD completion of $A + {\color{black}\varepsilon(G)}I_n \in \Sigma(\hat{G})$.
    
    Now, we consider the case $b > 0$, where we will need an additional lemma.  In this case, we will denote by $A / v^*$ the partially specified matrix defined by $Q - b^{-1}\pi_{G}(cc^{\intercal}) \in \Sym(G)$. We call this the Schur complement of the partially specified matrix $A$ with respect to the $1\times 1$ submatrix corresponding to  $v^*$.
    
   \begin{lemma}\label{lem:schur}
       If $A$ is in $\mathcal{P}(\hat{G})$ and $b \neq 0$, then $A / v^* \in \mathcal{P}(G)$.
   \end{lemma} 
   This lemma states that every point in $\mathcal{P}(\hat{G})$ decomposes into the sum of a rank 1 PSD-completable partial matrix and a partial matrix in $\mathcal{P}(G)$.
   \begin{proof}[Proof of Lemma \ref{lem:schur}]
       This follows from the fact that Schur complements preserve the class of positive semidefinite matrices (see \cite[A.5.5]{cvx_opt}).
       
       To show that $A / v^*$ is in $\mathcal{P}(G)$, we need to show that for each clique $K \subseteq G$, $\pi_K(A / v^*)$ is PSD.
       
       Given a clique $K$ in $G$, we can consider the subgraph of $\hat{G}$ induced by $V(K) \cup \{v^*\}$, which will be a clique in $\hat{G}${\color{black}, since $v^*$ is connected to all vertices of $G$}. Because $A \in \mathcal{P}(\hat{G})$, we have that $\pi_{K \cup \{v^*\}}(A)$ is PSD.
       
       Now, we see that 
       \[
       \pi_{K \cup \{v^*\}}(A) = 
        \begin{pmatrix}
            \pi_K(Q)& \pi_K(c) \\
            \pi_K(c) & b
        \end{pmatrix},
       \]
       where $\pi_K(c)$ is the projection of $c$ onto the subspace of $\R^n$ corresponding to the vertices in $K$.  {\color{black} The Schur complement of this fully specified matrix is then
    \[
        \pi_{K \cup \{v^*\}}(A) / v^* = \pi_K(Q) - b^{-1}\pi_K(c) \pi_K(c)^{\intercal}.
    \]
    Because $\pi_{K \cup \{v^*\}}(A)$ is positive semidefinite, its Schur complement is positive semidefinite.
    
    Compare this to $\pi_K(A / v^*)$, which is
    \[
        \pi_{K}(A / v^*) = \pi_K(Q - b^{-1}cc^{\intercal}) =  \pi_K(Q) - b^{-1}\pi_K(cc^{\intercal}).
    \]
    Note that the entries of $\pi_K(cc^{\intercal})$ satisfy
    \[
        \pi_K(cc^{\intercal})_{ij} = c_ic_j = (\pi_K(c)_i)(\pi_K(c)_j),
    \]
    so that $\pi_K(cc^{\intercal}) = \pi_K(c)\pi_K(c)^{\intercal}$.
       }
       
       We then see that \[
       \pi_K(Q - b^{-1}cc^{\intercal}) = 
       \pi_K(Q) - b^{-1}\pi_K(c)\pi_K(c)^{\intercal}
       \]
       is the Schur complement of the fully specified matrix $\pi_{K \cup \{v^*\}}(A)$. {\color{black} Because the Schur complement of a fully specified PSD matrix is PSD, $\pi_K(A / v^*)$ is therefore PSD for every clique $K$ of $G$, as desired}.
   \end{proof}

\emph{Proof of Theorem \ref{thm:join} continued.}   Now, we see that if $A \in \mathcal{P}(\hat{G})$, then $A$ is of the form 
   \[
        A = 
        \begin{pmatrix}
            A / v^* & 0 \\
            0 & 0
        \end{pmatrix} + 
        \pi_{\hat{G}}
        \begin{pmatrix}
            b^{-1}cc^{\intercal} & c \\
            c^{\intercal} & b
        \end{pmatrix} = 
        \begin{pmatrix}
            A / v^*& 0 \\
            0 & 0
        \end{pmatrix} + 
        \pi_{\hat{G}}\left(
        \begin{pmatrix}
            \frac{c}{\sqrt{b}} \\ \sqrt{b}
        \end{pmatrix}
        \begin{pmatrix}
            \frac{c^{\intercal}}{\sqrt{b}} & \sqrt{b}
        \end{pmatrix}\right).
   \]
   Since $A / v^* \in \mathcal{P}(G)$, and $\tr(A / v^*) \le 1$, we obtain that $A / v^* + \varepsilon(G)I_{n-1} \in \Sigma(G)$:
   \[
        A + \varepsilon(G)I_n =
        \begin{pmatrix}
            A / v^* + {\color{black}\varepsilon(G)}I_{n-1}& 0 \\
            0 & 0
        \end{pmatrix} + 
        \pi_{\hat{G}}\left(
        \begin{pmatrix}
            \frac{c}{\sqrt{b}} \\ \sqrt{b}
        \end{pmatrix}
        \begin{pmatrix}
            \frac{c^{\intercal}}{\sqrt{b}} & \sqrt{b}
        \end{pmatrix}\right) + 
        \begin{pmatrix}
            0 & 0 \\
            0 & {\color{black}\varepsilon(G)}
        \end{pmatrix}.
   \]
   
This is the sum of three matrices which are each in $\Sigma(\hat{G})$, and so the sum is in $\Sigma(\hat{G})$, as desired.
\end{proof}
\begin{remark}
This proof shows something a little stronger than the statement about conical distance: it in fact shows that 
every element $A \in \mathcal{P}(\hat{G})$ can be decomposed as $A = B + C$, where $B, C \in \Pos(\hat{G})$, $B$ can be completed to a rank 1 PSD matrix, and $C$ satisfies $C_{v^*v^*} = 0$.
\end{remark}

Iterating the previous lemma $k$ times gives us the following:
\begin{theorem}
    Let $K_k$ be a complete graph of size $k$.  Then $\varepsilon(G \land K_k) = \varepsilon(G)$.
\end{theorem}

\section{Rank and an Exact Conical Distance for Cycle Graphs}
\subsection{Rank of an Optimal Completion}%
\label{sub:rank_of_an_optimal_completion}

Define the conical distance \textbf{at a positive partially specified matrix} $A \in \mathcal{P}(G)$ as
\[
    \varepsilon_G(A) = \min \{\varepsilon : A + \varepsilon I \in \Sigma(G)\}\change{.}
\]

First notice that for fixed $A$, this can be computed in terms of the following semidefinite program:

\begin{equation*}
\begin{aligned}
    \varepsilon_G(A) =\;&\text{min} & &\varepsilon\\
                        &\text{such that }& &\forall \{i,j\}\in E(G): \; (M - \varepsilon I_n)_{i,j} = A_{i,j}\\
                        &&&  M \succeq 0\change{.}
\end{aligned}
\end{equation*}

We see that the feasible set is defined only by equations and the positive semidefiniteness constraints, and moreover that by choosing $\varepsilon$ large enough, there will be positive definite feasible points. By Slater's condition (\cite[theorem 2.15]{BPT}), we obtain strong duality:
\begin{equation*}
\begin{aligned}
    \varepsilon_G(A) =\;&\text{max} & & \sum_{\{i,j\} \in E(G)} A_{i,j} Y_{i,j} \\
                    & \text{such that } & &\forall \{i,j\} \not \in E(G):\;Y_{i,j} = 0\\
                     &  &&\tr(Y) = 1\\
                     &  &&Y \succeq 0\change{.}
\end{aligned}
\end{equation*}

The convex cone associated with the feasible set of this program, 
\[
    \Sigma^*(G) = \{Y : Y \succeq 0, \forall i,j \not \in E(G): Y_{i,j} = 0\}
\]
is the dual cone to $\Sigma(G)$.

Now, if $M^*$ is the positive semidefinite completion of $A + \varepsilon_G(A)I_n$ obtaining the optimum in the primal SDP, and if $Y^*$ optimizes the dual SDP, and then by complementary slackness,
\[
    Y^* M^* = 0 .
\]

Since $M^*$ is positive semidefinite, all nonzero eigenvectors of $M^*$ must lie in the kernel of $Y^*$, and in particular, $\text{rank}(Y^*) + \text{rank}(M^*) \le n$.

We note that \cite{free_resolution} gives a bound on the rank of any extreme point in the dual cone $\Sigma^*(G)$. We restate this theorem here. Define the \textbf{chordal girth} of $G$ to be the \change{number of vertices in the }smallest induced cycle in $G$ with at least 4 vertices, or $\infty$ if $G$ is chordal.

\begin{theorem}
    If $G$ is a graph whose chordal girth is $g$, and $Y$ is an extreme point of the dual cone $\Sigma^*(G)$, so that $Y$ is not in $\mathcal{P}^*(G)$, then the rank of $Y$ is at least $g-2$.
\end{theorem}

Applying this theorem directly to $M^*$ and $Y^*$ in the previous example gives:

\begin{theorem}
    Let $g$ be the chordal girth of $G$, suppose that $\varepsilon_G(A) > 0$, and let $M^*$ be a PSD completion of $A + {\color{black}\varepsilon_G(A)}I_n$. Then the rank of $M^*$ is at most $n - g + 2$.
\end{theorem}

In particular, if $G$ is a cycle, we see that this rank is at most 2, which will be useful in one of our next results.
\subsection{Exact Conical Distance for the Cycle}
\label{sub:exact_expansion_ratio_for_the_cycle}

The simplest kind of graphs for which we can consider the conical distance question, besides the chordal graphs, is the family of cycle graphs. In fact, we can compute the conical distance for these graphs exactly.

\begin{theorem}\label{thm:cycle}
        For $n > 3$,
        $\varepsilon(C_n) = \frac{1}{n}\left( \frac{1}{\cos(\frac{\pi}{n})} - 1\right)$.
\end{theorem}

Before we prove the main theorem in this section, we will need some lemmas.

First, we show that matrices in $\mathcal{P}(C_n)$ which maximize conical distance to $\Sigma(C_n)$ can be reduced to a certain normal form.
\begin{lemma}
    The function $\varepsilon_{C_n}(A)$ is maximized in $\bar{\mathcal{P}}({C_n})$ by a partially positive matrix $A$ such that every fully specified $2\times 2$ minor of $A$ is singular, and $A$ has exactly one negative entry.
\end{lemma}

\begin{proof}
First, we note that $\varepsilon_{C_n}(A)$ is concave as a function of $A$, since if $A$ and $Q$ are in $\bar{\Pos}({C_n})$, and $\alpha, \beta \ge 0$ are such that $\alpha + \beta = 1$, then
\[
    \alpha A + \beta Q + (\alpha\varepsilon_{C_n}(A) + \beta\varepsilon_{C_n}(Q)) I_n = 
    \alpha (A + \varepsilon_{C_n}(A)I_n) + \beta(Q + \varepsilon_{C_n}(Q) I_n).
\]
is the sum of two PSD completable matrices, so that it is PSD completable. Thus this function is maximized at an extreme point.

Note that for $n > 3$, the only cliques in $C_n$ are the edges, so we only need to check that all of the $2 \times 2$ fully specified submatrices are PSD.

Consider any extreme ray $A$ in $\mathcal{P}({C_n})$.
For each edge $e = \{i,j\} \in E(G)$, consider $\det(\pi_e(A))$.
If $\det(\pi_e(A)) > 0$, then we see that if we replace $A_{i,j}$ by $A_{i,j} \pm \delta$ for $\delta$ small enough, then the determinant of this block will still be positive, and the $2\times 2$ block of $A$ would be PSD, and because no other clique in $G$ contains the edge $\{i,j\}$, we see that $A$ will also remain inside $\mathcal{P}$.
This contradicts the fact that $A$ is an extreme ray.
Thus, for each edge $e = \{i,j\}$, $\det(\pi_e(A)) = 0$.

That is, for each $i,j$ so that $\{i,  j\} \in E(C_n)$, $A_{ii}A_{jj} - A_{ij}^2 = 0$. Rewriting, this implies that $A_{ij} = \pm \sqrt{A_{ii}A_{jj}}$. That is, the diagonal entries of an extreme point will determine the off-diagonal entries up to a choice of signs. 

Next, we note that if $D$ is a diagonal matrix where all of the diagonal entries are either $1$ or $-1$, then $D$ is unitary, and hence, conjugating a matrix by $D$ preserves its eigenvalues. Given $A$, an extreme ray in $\mathcal{P}({C_n})$, we can conjugate it by an appropriate diagonal matrix $D$ so that $A$ has a minimal number of negative entries. There are two cases: either $A$ can be conjugated so that all of its entries are nonnegative, or so that exactly one pair of entries is negative. If $A$ can be conjugated so that all of its entries are nonnegative, then in fact, it is the projection of a rank 1 PSD matrix, and so, it is PSD completable. The only important case then is the case when $A$ has exactly one pair of negative entries, and we will call this the \emph{normal form} of an extreme ray $A$.
\end{proof}

\begin{example}
Consider the case of extreme points in $\Pos(C_4)$. We note that since each fully specified submatrix is rank 1, the form of such a partially matrix is:
\[
    \begin{pmatrix}
        m_{11} & \pm \sqrt{m_{11}m_{22}} & ? &  \pm \sqrt{m_{11}m_{44}}\\
        \pm \sqrt{m_{11}m_{22}} & m_{22} & \pm \sqrt{m_{22}m_{44}} & ?\\
         ? & \pm \sqrt{m_{22}m_{33}} & m_{33} & \pm \sqrt{m_{33}m_{44}}\\
         \pm \sqrt{m_{11}m_{44}} & ? & \pm \sqrt{m_{33} m_{44}} & m_{44}
    \end{pmatrix}.
\]

If we conjugate by \change{an }appropriate $\pm 1$ diagonal matrix, we can bring this into the form
\begin{gather*}
    D
    \begin{pmatrix}
        m_{11} & \pm \sqrt{m_{11}m_{22}} & ? &  \pm \sqrt{m_{11}m_{44}}\\
        \pm \sqrt{m_{11}m_{22}} & m_{22} & \pm \sqrt{m_{22}m_{44}} & ?\\
         ? & \pm \sqrt{m_{22}m_{33}} & m_{33} & \pm \sqrt{m_{33}m_{44}}\\
         \pm \sqrt{m_{11}m_{44}} & ? & \pm \sqrt{m_{33} m_{44}} & m_{44}
    \end{pmatrix}
    D = \\
    \begin{pmatrix}
        m_{11} & \sqrt{m_{11}m_{22}} & ? &  \pm\sqrt{m_{11}m_{44}}\\
        \sqrt{m_{11}m_{22}} & m_{22} & \sqrt{m_{22}m_{44}} & ?\\
         ? & \sqrt{m_{22}m_{33}} & m_{33} & \sqrt{m_{33}m_{44}}\\
        \pm  \sqrt{m_{11}m_{44}} & ? & \sqrt{m_{33} m_{44}} & m_{44}
    \end{pmatrix}.
\end{gather*}
Here, $D$ is the diagonal matrix,
\[
    D =     \begin{pmatrix}
        \pm 1 & 0 & 0 & 0\\
        0 & \pm 1 & 0 & 0\\
        0 & 0 & \pm 1 & 0\\
        0 & 0 & 0 & \pm 1\\
    \end{pmatrix}.
\]

    If the $(1,4)$ entry is positive, then it is clear that the resulting partial matrix is completable to a rank 1 PSD matrix.
\end{example}

We will identify $n+1$ with $1$, so that we can write expressions such as $\sum_{i=1}^n A_{ii+1}$, and this will mean that we sum up all of the off-diagonal entries in $A$.

Also, let $\arccos(x) : [-1,1] \rightarrow [0, \pi]$ be the inverse of the $\cos(\theta)$ function defined on this interval.

We can now show a
necessary and sufficient
condition for $A \in \Sym(C_n)$ to be completable to a rank 2 PSD matrix.
\begin{lemma}\label{lemma:equiv_cond}
    If $A \in \mathcal{P}(C_n)$, then $A$ has a rank 2 PSD completion if and only if there are some $a_i \in \{1, -1 \}$ so that
    \begin{align*}
        \sum_{i=1}^n a_i \arccos\left(\frac{A_{ii+1}}{\sqrt{A_{ii}}\sqrt{A_{i+1i+1}}}\right) = 2k\pi,\tag{*}\\
    \end{align*}
    for some integer $k$.
\end{lemma}
\begin{proof}
    We make a sequence of reductions to prove the result.

    In terms of vector arrangements, $A$ is completable to a PSD rank 2 matrix if and only if there are vectors $v_1, \dots, v_n \in \R^2$ so that for each $i$
    \begin{align*}
        \|v_i\|^2 &= A_{ii},\\
        \langle v_{i}, v_{i+1} \rangle &= A_{ii+1}.
    \end{align*}

    For the sake of notation, let $\bar{A}_i = \frac{A_{ii+1}}{\sqrt{A_{ii}}\sqrt{A_{i+1i+1}}}$. 
    If we renormalize these equations to make the $v_i$ lie on the unit circle, we can equivalently ask for $v_i \in \R^2$ so that
    \begin{align*}
        \|v_i\|^2 &= 1,\\
        \langle v_{i}, v_{i+1} \rangle &= \bar{A}_i\change{.}
    \end{align*}
    In this case, we will think of $a_i \arccos(\bar{A}_i)$ as the angle between $v_i$ and $v_{i+1}$; the equation is equivalent to the condition that the sum of the angles between the vectors on the circle is a multiple of $2\pi$.
    
    Formally, for each $v_i \in \R^2$ so that $\langle v_i, v_i \rangle = 1$, we can express $v_i$ in polar coordinates. This implies that there are some $\theta_i$ so that
    \[
        v_i = (\cos(\theta_i), \sin(\theta_i)).
    \]
    In this case, these equations reduce to the equations
    \begin{align*}
        \cos(\theta_{i+1} - \theta_{i}) = \bar{A}_i.
    \end{align*}

    To see the necessity of the equation $(*)$, note that if there exist $\theta_i$ satisfying the previous equation, then using some basic facts about the $\cos$ function, there exist some $a_i \in \{-1, 1\}$ and $\ell_i \in \mathbb{Z}$ so that
    \begin{align*}
        \theta_{i+1} - \theta_{i} = a_i \arccos(\bar{A}_i) + 2\pi \ell_i.
    \end{align*}
    Letting $k = -\sum_{i=1}^n \ell_i$, this implies that
    \begin{align*}
        \sum_{i=1}^n(\theta_{i+1} - \theta_{i}) &= \sum_{i=1}^n (a_i \arccos(\bar{A}_i) + 2 \pi \ell_i)\\
                                                &= \sum_{i=1}^n a_i \arccos\left(\frac{A_{ii+1}}{\sqrt{A_{ii}}\sqrt{A_{i+1i+1}}}\right) - 2k\pi\\
                                                &= 0.
    \end{align*}
    which clearly implies equation $(*)$.

    To see the sufficiency of equation $(*)$, suppose that there exist $a_i$ and some $k$ so that equation $(*)$ holds.
    Then, set $\theta_1 = 0$, and for each $1 \le i < n$, set
\[
        \theta_{i+1} = \theta_{i} + a_i \arccos(\bar{A}_i).
    \]
    Then, clearly, for $i < n$, we have the desired result that
    \[
        \cos(\theta_{i+1} - \theta_{i}) = \cos(a_i \arccos(\bar{A}_i)) = \bar{A}_i\change{,}
    \]
    and for $i = n$, we see that 
    \begin{align*} 
          \theta_n &= \sum_{i=1}^{n-1} a_i \arccos(\bar{A}_i)\\
                   &= \sum_{i=1}^{n-1} a_i \arccos\left(\frac{A_{ii+1}}{\sqrt{A_{ii}}\sqrt{A_{i+1i+1}}}\right)\\
                   &= 2k\pi - a_n\arccos\left(\frac{A_{1n}}{\sqrt{A_{11}}\sqrt{A_{nn}}}\right).
    \end{align*}
    Therefore,
    \begin{align*} 
        \cos(\theta_1 - \theta_n) &= \cos\left(2k\pi - a_n\arccos\left(\frac{A_{1n}}{\sqrt{A_{11}}\sqrt{A_{nn}}}\right)  \right)\\
                                  &= \bar{A_n},\\
    \end{align*}
    as we desired.

\end{proof}
The previous lemma has particular application to a partial matrix in normal form.
\begin{lemma}\label{lemma:simpler_equiv_cond}
    If $A \in \mathcal{P}(C_n)$, and $A$ is in normal form, then $A + \varepsilon I_n$ has a rank 2 PSD completion if
    \[
        \sum_{i=1}^n \arccos\left(\frac{\sqrt{A_{ii}A_{i+1i+1}}}{\sqrt{A_{ii} + {\color{black}\varepsilon}}\sqrt{A_{i+1i+1} + \varepsilon}}\right) = \pi.
    \]
\end{lemma}
\begin{proof}
    This follows from lemma \ref{lemma:equiv_cond}: after considering the sum
    \begin{align*}
        \sum_{i=1}^{n-1} \arccos\left(\frac{A_{ii+1}}{\sqrt{A_{ii} + {\color{black}\varepsilon}}\sqrt{A_{i+1i+1}+{\color{black}\varepsilon}}}\right) - 
        \arccos\left(\frac{A_{1n}}{\sqrt{A_{11} + {\color{black}\varepsilon}}\sqrt{A_{nn} + {\color{black}\varepsilon}}}\right) \\
        =
        \sum_{i=1}^{n-1} \arccos\left(\frac{\sqrt{A_{ii}A_{i+1i+1}}}{\sqrt{A_{ii} + {\color{black}\varepsilon}}\sqrt{A_{i+1i+1}+{\color{black}\varepsilon}}}\right) - 
        \arccos\left(-\frac{\sqrt{A_{11}A_{nn}}}{\sqrt{A_{11} + {\color{black}\varepsilon}}\sqrt{A_{nn} + {\color{black}\varepsilon}}}\right) \\
        =
        \sum_{i=1}^{n} \arccos\left(\frac{\sqrt{A_{ii}A_{i+1i+1}}}{\sqrt{A_{ii} + {\color{black}\varepsilon}}\sqrt{A_{i+1i+1}+{\color{black}\varepsilon}}}\right) - 
        \pi\\
        = 0 \change{.}
    \end{align*}
    where we have used the equation $\arccos(-x) = \pi - \arccos(x)$, and the hypothesis of the lemma.
    
\end{proof}
The last, somewhat mysterious fact we will need is the following:
\begin{lemma}\label{lemma:convexity}
    For any ${\color{black}\varepsilon} \ge 0$, the function $f_{\varepsilon}:\R^2_+ \rightarrow \R$ given by 
    \[
        f_{\varepsilon}(x, y) = \arccos(\frac{\sqrt{xy}}{\sqrt{x + {\color{black}\varepsilon}}\sqrt{y + {\color{black}\varepsilon}}})
    \]
    is convex.
\end{lemma}
\begin{proof}
    We prove this by computing the Hessian matrix of $f_{\varepsilon}$.
    \[
        H(f_{\varepsilon}) = 
        \begin{pmatrix}
            \pard{^2}{x^2}f_{\varepsilon}&
            \pard{^2}{x\partial y}f_{\varepsilon}\\
            \pard{^2}{x\partial y}f_{\varepsilon}&
            \pard{^2}{y^2}f_{\varepsilon}&
        \end{pmatrix}.
    \]
    This evaluates to
    \[
        \left(
\begin{array}{cc}
    \frac{{\color{black}\varepsilon}^2 y^2 \left({\color{black}\varepsilon}^2+{\color{black}\varepsilon} (5 x+y)+x (4 x+3 y)\right)}{4 ({\color{black}\varepsilon}+x)^{7/2} ({\color{black}\varepsilon}+y)^{3/2} (x y)^{3/2} \left(\frac{{\color{black}\varepsilon} ({\color{black}\varepsilon}+x+y)}{({\color{black}\varepsilon}+x) ({\color{black}\varepsilon}+y)}\right)^{3/2}} & -\frac{{\color{black}\varepsilon}^2}{4 ({\color{black}\varepsilon}+x)^{3/2} ({\color{black}\varepsilon}+y)^{3/2} \sqrt{x y} \left(\frac{{\color{black}\varepsilon} ({\color{black}\varepsilon}+x+y)}{({\color{black}\varepsilon}+x) ({\color{black}\varepsilon}+y)}\right)^{3/2}} \\
    -\frac{{\color{black}\varepsilon}^2}{4 ({\color{black}\varepsilon}+x)^{3/2} ({\color{black}\varepsilon}+y)^{3/2} \sqrt{x y} \left(\frac{{\color{black}\varepsilon} ({\color{black}\varepsilon}+x+y)}{({\color{black}\varepsilon}+x) ({\color{black}\varepsilon}+y)}\right)^{3/2}} & \frac{{\color{black}\varepsilon}^2 x^2 \left({\color{black}\varepsilon}^2+{\color{black}\varepsilon} (x+5 y)+y (3 x+4 y)\right)}{4 ({\color{black}\varepsilon}+x)^{3/2} ({\color{black}\varepsilon}+y)^{7/2} (x y)^{3/2} \left(\frac{{\color{black}\varepsilon} ({\color{black}\varepsilon}+x+y)}{({\color{black}\varepsilon}+x) ({\color{black}\varepsilon}+y)}\right)^{3/2}} \\
\end{array}
\right).
    \]
    Note that if ${\color{black}\varepsilon}, x, y \ge 0$, then the diagonal entries of this matrix are nonnegative.

Consider the Hessian determinant of this function if $x, y > 0$: 
\[
    \det(H(f_{\varepsilon})) = \frac{{\color{black}\varepsilon} ({\color{black}\varepsilon} (x+y)+3 x y)}{4 x y ({\color{black}\varepsilon}+x)^2 ({\color{black}\varepsilon}+y)^2 ({\color{black}\varepsilon}+x+y)}.
\]
This is also nonnegative on this domain.

These two facts about $H(f_{\varepsilon})$ are enough to determine that it is positive semidefinite, and so $f$ is convex.
\end{proof}

\begin{proof}[Proof of Theorem \ref{thm:cycle}]
    Fix some $A$ in normal form. We wish to show that there is some ${\color{black}\varepsilon} \le \frac{1}{n}(\frac{1}{\cos(\frac{\pi}{n})} - 1)$ so that $A + \varepsilon I_n$ is completable to a PSD matrix with rank 2. We want to apply the condition in lemma \ref{lemma:simpler_equiv_cond}. Consider the function 
    \[
        g(\varepsilon) = \sum_{i=1}^{n}\arccos\left(\frac{\sqrt{A_{ii}A_{i+1i+1}}}{\sqrt{A_{ii} + {\color{black}\varepsilon}}\sqrt{A_{i+1i+1} + {\color{black}\varepsilon}}}\right).
    \]
    Lemma \ref{lemma:simpler_equiv_cond} implies if $\varepsilon$ is such that $g(\varepsilon) = \pi$, then in fact there is a rank 2 PSD completion of $A + \varepsilon I_n$.
    We will show that $g(0) = 0$ and $g(\frac{1}{n}(\frac{1}{\cos(\frac{\pi}{n})} - 1)) \ge \pi$, so that by the intermediate value theorem, there must be $\varepsilon \in [0, \frac{1}{n}(\frac{1}{\cos(\frac{\pi}{n})} - 1)]$ so that $g(\varepsilon) = \pi$, yielding the result.

    First note that if $\varepsilon = 0$, then 
    \[
        \sum_{i=1}^{n}\arccos\left(\frac{\sqrt{A_{ii}A_{i+1i+1}}}{\sqrt{A_{ii} + {\color{black}\varepsilon}}\sqrt{A_{i+1i+1} + {\color{black}\varepsilon}}}\right) =
        \sum_{i=1}^{n}\arccos(1) = 0 < \pi\change{.}
    \]
    
    Now, we want to show that if $\varepsilon = \frac{1}{n}(\frac{1}{\cos(\frac{\pi}{n})} - 1)$, then $g(\varepsilon) \ge \pi$.  We use lemma \ref{lemma:convexity} and the fact that $\tr(A) = 1$ to see that
    \begin{align*}
        g(\varepsilon) &= n\sum_{i=1}^n \frac{1}{n}f_{\varepsilon}\left(A_{ii}, A_{i+1i+1}\right)\\
                       & \ge nf_{\varepsilon}\left(\frac{1}{n}\sum_{i=1}^n A_{ii}, \frac{1}{n}\sum_{i=1}^n A_{ii}\right)\\
                       &= nf_{\change{\varepsilon}}\left(\frac{1}{n}, \frac{1}{n}\right)\\
                       &= n\arccos(\frac{1}{1 + n\varepsilon})\\
                       &= n\arccos(\cos(\frac{\pi}{n}))\\
                       &= \pi,
    \end{align*}
    as desired.

    Thus, there is some $\varepsilon \le \frac{1}{n}(\frac{1}{\cos(\frac{\pi}{n})} - 1)$ satisfying the condition of lemma \ref{lemma:simpler_equiv_cond}.

    To see that this value of $\varepsilon$ is in fact attained for some $A$, we can use the matrix $A$ where $A_{ii} = \frac{1}{n}$ for each $i$, $A_{12} = -\frac{1}{n}$, and each other specified off-diagonal entry is $\frac{1}{n}$. This matrix can be verified to have $\varepsilon_G(A) = \frac{1}{n}(\frac{1}{\cos(\frac{\pi}{n})} - 1)$ using the cycle conditions found in \cite{BARRETT19933}.
\end{proof}

We note that asymptotically $\frac{1}{n}\left( \frac{1}{\cos(\frac{\pi}{n})} - 1 \right)$ is $O(\frac{1}{n^3})$.
\subsection{Relative volumes of Nonnegative and Sums of Squares Cones for Cycles}
It is worth noting that this bound on the conical distance is strong enough to give us a bound on the relative volumes of slices $\bar{\mathcal{P}}(C_n)$ and $\bar{\Sigma}(C_n)$. By using the connection of conical distance with expansion ratio we see that:
\[
    {\color{black}\left(\bar{\Sigma}(C_n) - \frac{1}{n}I_n\right) \subseteq \left(\bar{\mathcal{P}}(C_n) - \frac{1}{n}I_n\right) \subseteq (1+n\varepsilon(C_n))\left(\bar{\Sigma}(C_n) - \frac{1}{n}I_n\right).}
\]
Since the dimension of $\Sym(C_n)$ is $2n$,
\[
    {\color{black}\operatorname{Vol}(\bar{\Sigma}(C_n)) \leq \operatorname{Vol}(\bar{\mathcal{P}}(C_n)) \leq (1+n\varepsilon(C_n))^{2n}\operatorname{Vol}(\bar{\Sigma}(C_n)).}
\]
Now, note that $\varepsilon(C_n) = O(\frac{1}{n^3})$, so we have that 
\[
    {\color{black}(1+n\varepsilon(C_n))^{2n} \le e^{O(1 / n)},}
\]
which approaches 1 as $n$ goes to infinity.

\subsection{Cycle Completable Graphs}

The following characterization of cycle-completable graphs proved in  \cite{MR1342017} allows us to find exact conical distance for a larger class of graphs.

\begin{theorem}\label{thm:seri_p}
    Let $G$ be a graph, let $\Gamma(G)$ be the set of matrices $A$ so that for each induced cycle $C \subseteq G$, $\pi_C(A) \in \Sigma(C)$. Then $\Sigma(G) = \mathcal{P}(G) \cap \Gamma(G)$ iff $G$ is the clique sum of series parallel and chordal graphs.
\end{theorem}

We are now ready to extend the result on exact conical distance between $\mathcal{P}(G)$ and $\Sigma(G)$ from cycles to a larger class of graphs using our previous results and the above theorem.

\begin{theorem}\label{thm:main}
    Let $\mathcal{G}$ be the smallest class of graphs which is closed under the clique sum operation and the cone operation, and which contains all complete graphs and series parallel graphs. Then for any $G \in \mathcal{G}$, $\varepsilon(G) = \varepsilon(C_g)$, where $g$ is the chordal girth of $G$.
\end{theorem}

\begin{proof}

It suffices to show that $\varepsilon(G) = \varepsilon(C_g)$ for any series parallel graph $G$. The theorem then follows from applying Theorems \ref{thm:csum} and \ref{thm:join}.

 If $g$ is the chordal girth of $G$, then there is no cycle in $G$ with length greater than 3 and less than $g$.
 Moreover, note that the conical distance for $C_g$ is monotonic decreasing in $g$. 
 
If $C$ is any cycle contained in $G$, then $\varepsilon(C) \le \varepsilon(C_g)$. Thus, if $A \in \mathcal{P}(G)$, then for any cycle $C$ in $G$, $\pi_C(A + \varepsilon(C_g) I_n) = \pi_C(A) + \varepsilon(C_g)I_n \in \Sigma(C)$.Thus, $A + \varepsilon(C_g)I_n \in \mathcal{P}(G)$, and moreover, for all cycles $C$ in $G$, $\pi_C(A + \varepsilon(C_g)I_n) \in \Sigma(C)$, so by Theorem \ref{thm:seri_p}, we have that $A + \varepsilon(C_g)I_n \in \Sigma(G)$, as desired.
\end{proof}

It is noteworthy that the class of graphs described in \cite{MR1342017} has the wheels, $W_n = \hat{C_n}$, as its forbidden minors, yet the conical distance for the wheel graphs goes to 0 at the rate $O(\frac{1}{n^3})$ as $n$ grows, which implies small expansion ratio and good volume approximation for $\bar{\Sigma}(G)$ and $\bar{\mathcal{P}}(G)$.

\section{Applications to Semidefinite Programming}
Conical distance can be used to bound approximation errors in some semidefinite programs. Consider a semidefinite program of the form 

\begin{equation*}
\begin{aligned}
    \text{minimize} &&\langle B^0,X\rangle\\
    \text{such that } &&\tr(X) = 1,\\
                      &&\langle B^1,X\rangle = b_1,\\
                      &&\dots\\
                      &&\langle B^k,X\rangle = b_k,\\
                      &&  X \succeq 0\change{.}
\end{aligned}\tag{SDP}
\end{equation*}
We will say that this program is $G$-sparse, for some graph $G$ if for each $\ell$, $B^{\ell}_{ij} = 0$ for all $\{i,j\} \not \in E(G)$.

If this semidefinite program is $G$-sparse, then the semidefinite program is equivalent to one where the condition $X\succeq 0$ is replaced by $X \in \Sigma(G)$ without losing information.

Consider replacing the condition that $X \in \Sigma(G)$ with the condition that $X \in \mathcal{P}(G)$. 

\begin{equation*}
\begin{aligned}
    \text{minimize} &&\langle B^0,X\rangle\\
    \text{such that } &&\tr(X) = 1,\\
                      &&\langle B^1,X\rangle = b_1,\\
                      &&\dots\\
                      &&\langle B^k,X\rangle = b_k,\\
      &&  X \in \mathcal{P}(G)\change{.}
\end{aligned}\tag{SDP'}
\end{equation*}

The condition $X \in \mathcal{P}(G)$ can be defined while only considering the entries of $X$ which correspond to edges in $G$. If $G$ is a graph in our class $\mathcal{G}$, which has $o(n^2)$ edges, and has no small cycles of size greater than 3, then the resulting savings in terms of the number of variables needed to define the program can be large while not costing too much in terms of the approximation factor.

Let $\alpha$ be the value of (SDP), and $\alpha'$ be the value of (SDP')
\begin{theorem}\label{thm:SDP}
    If $\frac{1}{n} I_n$ is a feasible point of (SDP), and (SDP) is $G$-sparse, then 
    \[
        {\color{black}\alpha' \le \alpha \le \frac{1}{1+n\varepsilon(G)}\alpha' + \frac{\varepsilon(G)}{1+n\varepsilon(G)}\tr(B^0).}
    \]
\end{theorem}
\begin{proof}
Since $\Sigma(G) \subseteq \mathcal{P}(G)$, we have that $\alpha \ge \alpha'$.

Let $X'$ be an optimum point for (SDP'). By the definition of the conical distance, 
\[
    X' + \varepsilon(G)I_n \in \Sigma(G).
\]

Then because $\Sigma(G)$ is a cone, we have the rescaled equation,
\[
    \frac{1}{1+n\varepsilon(G)}X' + \frac{n\varepsilon(G)}{1+n\varepsilon(G)}\left(\frac{1}{n}I_n\right) \in \Sigma(G).
\]

Note that this is a convex combination of $X'$ and $\frac{1}{n}I_n$, and thus, it satisfies all of the linear equations for the feasible set for (SDP). The optimal value therefore satisfies
\[
    \alpha \le \left\langle B^0,\left(\frac{1}{1+\varepsilon(G)}X' + \frac{\varepsilon(G)}{n(1+\varepsilon(G))}I_n\right) \right\rangle = \frac{1}{1+n\varepsilon(G)}\alpha' + \frac{\varepsilon(G)}{1+n\varepsilon(G)}\tr(B^0).
\]
\end{proof}

A key example of this type of semidefinite program is the rescaled Goemans and Williamson SDP (see \cite{goemans1995improved}) for approximating MAX-CUT for a graph $G$.
For consistency with the rest of the paper, we have written this as a minimization problem and reversed the sign of the objective function, so that the result is a negative number.
In that case, the definition of the semidefinite program is
\begin{equation*}
\begin{aligned}
    \text{minimize} &&n\sum_{i,j\in E(G), i \neq j} X_{ij}\\
    \text{such that } &&\forall i, X_{ii} = \frac{1}{n}\\
    &&  X \succeq 0\change{.}
\end{aligned}
\end{equation*}
This semidefinite program satisfies the conditions for the theorem, and moreover, we see that $\tr(B^0) = 0$, so the result is the cleaner 
\[
    \alpha' \le \alpha \le \frac{1}{1+n\varepsilon(G)}\alpha'.
\]

In applications, the input graph $G$ is not \change{necessarily} chordal. Rather, given an input graph, one attempts to find a \textbf{chordal cover} of $G$, i.e. a chordal graph $C$ so that $G$ is contained in $C$. Any $G$-sparse semidefinite program then is also $C$-sparse, and the results of \cite{chordal} apply. However, finding a chordal graph $C$ which minimizes the number of edges added to $E(C) \setminus E(G)$ is known to be NP-complete, even to approximate \cite{CAO2020104514}. Thus, it may be valuable in practice to relax the chordal condition to a somewhat more general condition, and settle for an approximation.

\subsection{Properties of the class $\mathcal{G}$}\label{sec:G}
\label{sec:graph_properties}
Graph theoretic questions about the graph class $\mathcal{G}$ are of practical interest, and we will briefly give some results here. One important property is that $\mathcal{G}$ is closed under taking induced subgraphs.


A \textbf{clique separator} of a graph $G$ is a clique $K \subseteq G$ with the property that removing $K$ from $G$ increases the number of connected components of the graph $G$. If $K$ is such a clique separator of $G$, we can write $G = H_1 \oplus H_2$, where $H_1$ and $H_2$ are induced subgraphs of $G$, and $K = H_1 \cap H_2$.

\change{A graph $H$ with no clique separators will be referred to as an \textbf{atom}.} By repeatedly finding clique separators of the graph $G$, we can decompose $G$ into the clique sum of \change{atoms: $G = H_1 \oplus H_2 \oplus \dots H_k$, where each $H_i$ is an atom.  We will refer to the $H_i$'s as the \textbf{atoms of the decomposition}, and when the decomposition is left implicit, we will refer to the $H_i$ as atoms of $G$.}
It is clear then that any graph $G$ is the clique sum of its atoms.
Note that the atoms of a graph depend on a particular decomposition of the graph $G$, and cannot be uniquely determined from the graph $G$.

We can characterize the possible atoms of a graph $G \in \mathcal{G}$.

\begin{theorem}
    If $G$ is a graph, then $G \in \mathcal{G}$ if and only if for any decomposition of $G$ into atoms, all atoms are of the form $H \wedge K_n$, where $H$ is series parallel and $H \wedge K_n$ is the $n$-fold repeated coning over the graph $H$.
\end{theorem}
\begin{proof}
    Suppose that $G$ is a graph with a \change{decomposition into atoms, each of which is of the form} $H \wedge K_n$ where $H$ is a series parallel graph. Since $H$ is series parallel, it is contained in $\mathcal{G}$, and thus the repeated coning over $H$ is contained in $\mathcal{G}$. So, the atoms of $G$ are contained in $\mathcal{G}$. Because $G$ is the clique sum of its atoms, and $\mathcal{G}$ is closed under clique sums, we see that $G$ is contained in $\mathcal{G}$.

    On the other hand, suppose that $G \change{\in \mathcal{G}}$. Let $G_1 \oplus G_2  \oplus\dots\oplus G_k$ be a decomposition of $G$ into a clique sum of atoms. Note that the partial sum $G_1 \oplus G_2  \oplus\dots\oplus G_{k-1}$ is an induced subgraph of $G$, and therefore, is also an element of $\mathcal{G}$, and thus, inductively, it suffices to show that $G_k$ is of the form $H \wedge K_n$.

    By definition, $G_k$ is an atom, and is an induced subgraph of $G$. In particular, $G_k$ is an element of $\mathcal{G}$, and $G_k$ is either series parallel, chordal, the cone over a graph in $\mathcal{G}$, or a clique sum of two graphs in $\mathcal{G}$. Clearly, if $G_k$ is the clique sum of two graphs, then it is not an atom, and thus, we can exclude this case.

    If $G_k = \hat{J}$ for some $J \in \mathcal{G}$, then we can also assume that $J$ does not have a clique separator, since if $J = J_1 \oplus J_2$, then $G_k$ would be $\hat{J_1} \oplus \hat{J_2}$. Thus, we can iteratively remove cone \change{vertices} until we are left with a graph with no cone vertices. In that case, $G_k = H \wedge K_n$, where $H \in \mathcal{G}$ has no clique separators or cone vertices. Thus, $H$ must be either chordal or series parallel. \change{If $H$ were chordal, then $H$ would be the clique sum of cliques. Clearly, if $H$ is the clique sum of cliques, and also an atom, then $H$ must be a clique. Moreover, if $H$ is a clique that contains no cone vertices, then $H$ must be a single vertex, and thus also series parallel.}
    This shows the desired result.
\end{proof}

A result of Tarjan in \cite{tarjan1985decomposition} shows that it is possible to \change{find a decomposition of a graph $G$ into at most $n-1$ atoms} in $O(nm)$ time, where $n = |V(G)|$ and $m = |E(G)|$ \change{(assuming $m > 0$)}. A result of Valdes, Tarjan and Lawler in \cite{valdes1979recognition} also shows that it is possible to recognize series parallel graphs in linear time. It is thus clear that given a graph $G$, we can recognize if it is in $\mathcal{G}$ by first decomposing $G$ into atoms, then checking that each atom is the cone over a series parallel graph, which in total gives a $O(n^2\change{+nm})$ time algorithm for detecting graphs in $\mathcal{G}$.

It is then also easy to see that a graph in $\mathcal{G}$ does not have too many maximal cliques.

\begin{theorem}
    If $G \in \mathcal{G}$, then the number of maximal cliques in $G$ is $O(n(n+m))$\change{.}
\end{theorem}
\begin{proof}
    It is clear that if $G = H_1 \oplus H_2$, then the number of maximal cliques in $G$ is at most the sum of the number of maximal cliques in $H_1$ and $H_2$ separately. From above, we have that if $G \in \mathcal{G}$, then $G$ decomposes into the clique sum of at most $n$ graphs of the form $H \wedge K_k$, with $H$ series parallel.
    The coning operation does not change the number of maximal cliques, so it suffices to argue that a series parallel graph with at most $n$ vertices and $m$ edges, has at most $O(n+m)$ maximal cliques. This is clear from the characterization that a series parallel graph is a subgraph of a \textbf{partial-3-tree}, i.e. a  chordal graph with clique number at most 3, and it is clear that any chordal graph has at most $n$ maximal cliques. Thus, there are at most $O(n)$ cliques of size 3 in a series parallel graph, at most $m$ cliques of size $2$ (edges), and at most $O(n)$ cliques of size 1 (vertices). This implies the result.
\end{proof}
\begin{remark}
    It is likely possible, with a slightly more sophisticated accounting, that the number of maximal cliques for \change{a graph in $\mathcal{G}$} is in fact $O(n)$.
\end{remark}

As in the chordal case, it is of interest to find the smallest number of edges that can be added to a given graph, so that the resulting graph is in the class $\mathcal{G}$.  The authors are not aware of any results which indicate that this problem is asymptotically more tractable than completing to a chordal graph, though clearly fewer edges are required. It is likely that algorithms for finding small chordal completions can be modified to find small completions for graphs in our class.

\bibliographystyle{plain}
\bibliography{main}
\end{document}